\documentclass[11pt]{article}
\usepackage{}

\usepackage{amsmath,amsthm,verbatim,amssymb,amsfonts,amscd,graphicx, float}

\usepackage[colorlinks]{hyperref}

\usepackage{bm}
\usepackage{tikz}
\usetikzlibrary{arrows,shapes,positioning}
\usetikzlibrary{decorations.markings}
\tikzstyle arrowstyle=[scale=1]
\tikzstyle directed=[postaction={decorate,decoration={markings, mark=at position 0.75 with {\arrow[arrowstyle]{stealth}}}}]
\tikzstyle redirected=[postaction={decorate,decoration={markings, mark=at position 0.35 with {\arrow[arrowstyle]{stealth}}}}]

\newtheorem{theorem}{Theorem}[section]

\newtheorem{problem}[theorem]{Problem}
\newtheorem{lemma}[theorem]{Lemma}
\newtheorem{observation}[theorem]{Observation}

\newtheorem{proposition}[theorem]{Proposition}

\newtheorem{question}[theorem]{Question}

\DeclareMathOperator{\Z}{\mathbb{Z}}

\begin{document}

\title{Integer Flows on  Cyclic Covers of Graphs}
\author{
Rong Luo\thanks{Department of Mathematics, West Virginia University, Morgantown, WV 26506-6310, USA. 
Email address: rluo@mail.wvu.edu.}
~~ 
Yuying Ma\thanks{Department of Mathematics and Statistics, Georgia State University, Atlanta, GA 30303, USA.
Email address: yma29@gsu.edu.}
~~
D. Christopher Stephens\thanks{Department of Mathematical Sciences, Middle Tennessee State University, Murfreesboro, TN 37132, USA. Email address: Chris.Stephens@mtsu.edu.}
~~
Dong Ye\thanks{Department of Mathematical Sciences, Middle Tennessee State University, Murfreesboro, TN 37132, USA. Email address: dong.ye@mtsu.edu.}
}
\date{\relax}
\maketitle

\begin{abstract}
Graph covers preserve the local structure of a graph while allowing its global structure to change substantially. In this paper, we study how such changes in cyclic covers affect nowhere-zero flow properties.
For a fixed integer $\ell \ge 2$, 
we show that for every bridgeless graph $G$, there exists an $\ell$-fold cyclic cover of $G$ that admits a nowhere-zero $3$-flow if $\ell$ is even and a nowhere-zero $4$-flow if $\ell$ is odd. These bounds are best possible in general. 
For odd $\ell$, we further investigate which bridgeless graphs admit an $\ell$-fold cyclic cover with a nowhere-zero $3$-flow. 

\vspace{0.3cm}
\end{abstract}

\section{Introduction}

All graphs in this paper are finite and loopless, but may have parallel edges, unless otherwise specified. For terminology and notation not defined here, we follow~\cite{bondy2011}. Let $G$ be a graph and let $D$ be an orientation of $G$. A map $\phi:E(G) \to \Z$ together with the orientation $D$  is an \emph{integer flow} if, for every $v \in V(G)$, 
$$\sum_{e\in E^+_D(v)}\phi(e) = \sum_{e\in E^-_D(v)}\phi(e),$$
 where $E^+_D(v)$ and $E^-_D(v)$ are the sets of edges of $G$ whose tails and heads in $D$ are $v$, respectively. This condition is known as \emph{Kirchhoff's law}. The flow is \emph{nowhere-zero} if $\phi(e) \ne 0$ for every $e \in E(G)$. For an integer $k \ge 2$, \emph{a nowhere-zero $k$-flow} is an integer flow satisfying $1 \leq  |\phi(e)| < k$ for every $e \in E(G)$. Reversing the direction of an edge and negating its flow value does not affect Kirchhoff's law.

Nowhere-zero flows were introduced by Tutte in connection with graph embeddings and map colorings~\cite{tutte1949}. This viewpoint provides a natural extension of classical map-coloring problems from planar graphs to arbitrary graphs~\cite{Tutte1954}. In particular, the Four Color Theorem is equivalent to the statement that every bridgeless planar graph admits a nowhere-zero 4-flow. Tutte~\cite{Tutte1954} conjectured that every bridgeless graph admits a nowhere-zero 5-flow, and Seymour~\cite{Seymour1981} proved that every bridgeless graph admits a nowhere-zero 6-flow. Thus, a fundamental problem in flow theory is to determine the values of $k$ for which bridgeless graphs admit a nowhere-zero $k$-flow.

Graph coverings arise naturally from the theory of covering spaces by viewing graphs as one-dimensional complexes (cf. Chapter~2 in \cite{grosstucker1987}). Voltage graphs, introduced by Gross~\cite{gross1974} and further developed by Gross and Tucker~\cite{grosstucker1977,grosstucker1987}, provide an algebraic framework for constructing and describing graph coverings. In particular, voltage assignments in $\mathbb Z_{\ell}$, the cyclic group of integers modulo $\ell$, give rise to $\ell$-fold cyclic covers, which are the main objects studied in this paper. 

Graph covers provide a natural setting in which to study this flow question from a different perspective. They preserve the local structure of the base graph, for example, the degree of each lifted vertex, while its global structure may change substantially. Since nowhere-zero flows depend strongly on such global properties, it is natural to ask whether a suitable cover can improve the flow properties of a graph.

We now recall the construction of $\ell$-fold cyclic covers. 
Let $G$ be a graph, and fix an arbitrary reference orientation of its edges. For each edge $e \in E(G)$, let $t(e)$ and $h(e)$ denote the \emph{tail} and \emph{head} of $e$, respectively. For an integer $\ell \ge 2$, a \emph{$\Z_\ell$-voltage assignment} on $G$ is a map $\alpha: E(G) \to \Z_\ell$. When the orientation of an edge $e$ is  reversed, its voltage is changed into $-\alpha(e)$. The \emph{$\ell$-fold cyclic cover} of $G$ derived from $\alpha$, denoted by $\widetilde G := G(\alpha,\ell)$, is a graph with vertex set $V(\widetilde G) = V(G) \times \Z_\ell$ and edge set $E(\widetilde G) = E(G) \times \Z_\ell$, where, for each $e \in E(G)$ and $i \in \Z_\ell$, the lifted edge $(e,i)$ joins $(t(e),i)$ and $(h(e),i+\alpha(e))$. Note that the derived cyclic cover may not be connected.  Each lifted vertex $(v,i)$ of $\widetilde G$ has the same degree as $v$, and hence $\delta(\widetilde G) = \delta(G)$ and $\Delta(\widetilde G) = \Delta(G)$. Although the construction uses a reference orientation, the choice is auxiliary and does not affect the resulting cyclic cover. Therefore, throughout the paper, we fix one arbitrary reference orientation and regard each edge of $G$ as carrying this orientation whenever no confusion arises.

Cyclic covers preserve bridgelessness, since in any cyclic cover of a graph $G$, the subgraph corresponding to a cycle of $G$ is a disjoint union of cycles.
Consequently, by Seymour's $6$-flow theorem, every cyclic cover of a bridgeless graph admits a nowhere-zero $6$-flow. This leads to the following question.

\begin{question}\label{ques}
    Let $G$ be a bridgeless graph and let $\ell \ge 2$ be an integer. What is the smallest integer $k$ with  $2 \leq k \le 6$ for which there exists a $\Z_\ell$-voltage assignment $\alpha$ such that $G(\alpha, \ell)$ admits a nowhere-zero $k$-flow?
\end{question}

The case $k = 2$ is completely determined by degree parity. A graph admits a nowhere-zero $2$-flow if and only if every vertex has even degree. Since cyclic covers preserve vertex degrees, a graph $G$ admits a nowhere-zero $2$-flow if and only if every cyclic cover of $G$ does. In other words, if $G$ has a vertex of odd degree, no cyclic cover of $G$ admits a nowhere-zero $2$-flow. 

The next cases, $k = 3$ and $k = 4$, are therefore of particular interest. Related questions have been studied for tensor products of graphs. Zhang, Zheng, and Mamut~\cite{ZZM2007} investigated nowhere-zero $3$-flows in tensor products and obtained sufficient conditions for their existence. Of particular relevance to cyclic covers, Dong, Tay, and Koh~\cite{DTK2014} studied the tensor product $G\times K_2$, which can be viewed as the $2$-fold cyclic cover of $G$ derived from the voltage assignment $\alpha(e) = 1$ for $e \in E(G)$. For graphs satisfying $2 \le \delta(G) \le \Delta(G) \le 3$, they characterized when $G \times K_2$ admits a nowhere-zero $3$-flow. Nowhere-zero flows have also been studied in Cartesian products and bundles of graphs~\cite{IS2003, ROLLOVA2012, SZ2005}. Rollov\'{a} and \v{S}koviera~\cite{ROLLOVA2012} proved that every Cartesian bundle of two graphs without isolated vertices admits a nowhere-zero $4$-flow. An $\ell$-fold cyclic cover of $G$ can be viewed as a Cartesian bundle of $G$ and $\overline{K_\ell}$. However, since $\overline{K_\ell}$ consists entirely of isolated vertices, their result does not apply.

The following is our main result, which gives a sharp general answer to Question~\ref{ques}.

\begin{theorem}\label{thm:t-3-4}
Let $\ell \ge 2$ be an integer and let $G$ be a bridgeless graph. Then, the following hold:
\begin{itemize}
\setlength\itemsep{-0.15cm}
    \item[(i)] If $\ell$ is even, there exists an $\ell$-fold cyclic cover of $G$ admitting a nowhere-zero $3$-flow.
    \item[(ii)] If $\ell$ is odd, there exists an $\ell$-fold cyclic cover of $G$ admitting a nowhere-zero $4$-flow.
\end{itemize}
\end{theorem}

Both bounds in Theorem~\ref{thm:t-3-4} are best possible in general. The value $3$ for even $\ell$ cannot be reduced to $2$ for graphs with a vertex of odd degree. For odd $\ell$, there exist bridgeless graphs for which no $\ell$-fold cyclic cover admits a nowhere-zero $3$-flow. 
We further investigate this phenomenon later. In particular, for alternating double wheels, we characterize when a $3$-fold cyclic cover admitting a nowhere-zero $3$-flow exists, and give an example for which no odd-fold cyclic cover admits such a flow. 

The remainder of the paper is organized as follows. In the next section, we first consider cubic graphs. In Section~\ref{sec:main-proof}, we prove Theorem~\ref{thm:t-3-4} and give an application to Cartesian bundles. We conclude the paper by investigating which graphs admit $\ell$-fold cyclic covers with nowhere-zero $3$-flows when $\ell$ is odd.

\section{Cyclic covers of cubic graphs}\label{sec:cubic}
In this section, we focus on cubic graphs. We treat the cases of even and odd $\ell$ separately. We begin with the even case.

The characterization on cubic graphs below is due to Tutte~\cite{tutte1949},
while the general regular bipartite case follows from~\cite[Exercise~9.7]{zhang1997}.

\begin{lemma}\label{lem:regular-bipartite}
Every $d$-regular bipartite graph with $d \ge 3$ admits
a nowhere-zero $3$-flow. Moreover, a cubic graph admits
a nowhere-zero $3$-flow if and only if it is bipartite.
\end{lemma}

\begin{theorem}\label{thm:regular-l-2}
Let $\ell \ge 2$ be an integer and  let $G$ be a $d$-regular graph with $d \ge 3$. 

\medskip \noindent
(1) If $G$ is  bipartite, then every  $\ell$-fold cyclic cover of $G$  admits a nowhere-zero $3$-flow.

\medskip \noindent
(2) If $\ell$ is even, then there exists an $\ell$-fold cyclic cover of $G$ that admits a nowhere-zero $3$-flow.
\end{theorem}
\begin{proof} (1)  Let $(A,B)$ be a bipartition of  $G$.  Then for every $\Z_\ell$-voltage assignment $\alpha$, the sets $A \times \Z_\ell$ and $B \times \Z_\ell$ form a bipartition of $G(\alpha,\ell)$.  Thus $G(\alpha,\ell)$ is bipartite. Therefore it admits a nowhere-zero $3$-flow by Lemma~\ref{lem:regular-bipartite}. 

\medskip
(2) Assume that $\ell$ is even. Define a $\Z_\ell$-voltage assignment $\alpha_\ell$ by
\[
\alpha_\ell(e)= \frac{\ell}{2} \quad \text{for every $e\in E(G)$.}
\]

For each $i \in \{0,1,\dots,\ell/2-1\}$, every lifted edge of $G(\alpha_\ell, \ell)$ with one endpoint in $V(G) \times \{i, i+\ell/2\}$ has its other endpoint in the same vertex set. Also, the subgraph of $G(\alpha_\ell, \ell)$  induced by $V(G) \times \{i, i+\ell/2\}$ is isomorphic to $G(\alpha_2,2)$. Thus, 
 $G(\alpha_\ell, \ell)$ is the disjoint union of $\ell/2$ copies of $G(\alpha_2,2)$. Therefore it is sufficient to show  that $G(\alpha_2,2)$ admits a nowhere-zero $3$-flow.

 Since  $\alpha_2(e) = 1$ for every edge $e\in E(G)$, every edge in   $G(\alpha_2, 2)$ has one endpoint in $ V(G) \times \{0\}$ and one in $ V(G) \times \{1\}$, so $G(\alpha_2, 2)$ is a bipartite graph with bipartition $V(G) \times \{0\}$ and $V(G) \times \{1\}$. Moreover, since $G$ is $d$-regular, $G(\alpha_2, 2)$ is also $d$-regular. 
Therefore, $G(\alpha_2, 2)$ is a $d$-regular bipartite graph with $d \ge 3$.
Applying Lemma~\ref{lem:regular-bipartite} to $G(\alpha_2, 2)$, we conclude that $G(\alpha_2, 2)$ admits a nowhere-zero $3$-flow.
\end{proof} 

We now give a complete characterization for cubic graphs.

\begin{theorem}\label{thm:t-3-cubic}
Let $\ell \geq 2$ be an integer and $G$ be a cubic graph. Then  there is an $\ell$-fold cyclic cover of $G$ admitting a nowhere-zero $3$-flow  if and only if  either $\ell$ is even or $G$ is bipartite. 
\end{theorem}
\begin{proof} The sufficiency directly follows Theorem~\ref{thm:regular-l-2}.

Now we prove  the necessity.   Suppose that  some $\ell$-fold cyclic cover $\widetilde G:= G(\alpha, \ell)$ admits a nowhere-zero $3$-flow.  We  show that either $\ell$ is even or $G$ is bipartite. Suppose  by contradiction that $\ell$ is odd and $G$ is nonbipartite.  Then $G$ contains an odd cycle $C$. Since the subgraph $\widetilde C := C(\alpha|_{E(C)}, \ell)$ is a disjoint union of cycles in $\widetilde G$ and $|\widetilde C| = \ell |C|$.  Since both $|C|$ and $\ell$ are odd, $|\widetilde C|$ is odd, so  at least one component of $\widetilde C$ is an odd cycle. Thus $\widetilde G$ is not bipartite. Since $G$ is cubic,  $\widetilde G$ doesn't admit a  nowhere-zero $3$-flow by Lemma~\ref{lem:regular-bipartite}, a contradiction.
\end{proof}

The preceding characterization shows that when $\ell$ is odd, a nonbipartite cubic graph has no $\ell$-fold cyclic cover admitting a nowhere-zero $3$-flow. We therefore study nowhere-zero $4$-flows. 

By a classical result of Tutte~\cite{TUTTE1966}, a cubic graph admits a nowhere-zero $4$-flow if and only if it is $3$-edge-colorable. Since every cyclic cover of a cubic graph is also cubic, finding an $\ell$-fold cyclic cover admitting a nowhere-zero $4$-flow is equivalent to finding a $3$-edge-colorable $\ell$-fold cyclic cover. 

\begin{theorem}\label{thm:t-4-cubic}
    Let $\ell \ge 3$ be an odd integer and let $G$ be a bridgeless cubic graph. There exists an $\ell$-fold cyclic cover of $G$ that is $3$-edge-colorable and hence admits a nowhere-zero $4$-flow.
\end{theorem}

\begin{proof}
By Petersen's theorem~\cite{Petersen1891}, the bridgeless cubic graph $G$ admits a perfect matching $M$. Since $G$ is cubic, $G-M$ is a $2$-factor, which consists of vertex-disjoint cycles. 
Let $\mathcal{C}_o$ denote the set of odd cycles of $G-M$. 
For each $C \in \mathcal{C}_o$, choose an edge $f_C\in E(C)$, and set $F:=\{f_C\mid C \in \mathcal{C}_o\}$. Then, $F$ is a matching in $G$. Let $H:= G-F$. Then $H-M$ consists of vertex-disjoint even cycles and paths, where each path has an even number of edges. Thus, $H$ has a proper $3$-edge-coloring $\varphi_H: E(H) \to [3]:=\{1,2,3\}$ such that 

\noindent
$\bullet$ $\varphi_H(e) = 3$  if $e\in M$,

\noindent
$\bullet$ each component of $H-M$ is colored with $1$ and $2$ alternatively,

\noindent
$\bullet$ the two leaf edges  of each path component  in $H-M$ are colored with different colors since each path has an even number of edges.

For each vertex $x \in V(H)$, let $\overline{\varphi_H}(x)$ denote the set of colors in $[3]$ not used by any edge incident with $x$. For each $f_C\in F$, denote these endpoints  of $f_C$ by $v_C$ and $u_C$, respectively, so that  the edge $f_C$ is oriented from $v_C$ to $u_C$.   Then, we can choose $\varphi_H$ such that  
$$
1 \in \overline{\varphi_H}(v_C) \quad\text{and}\quad 2 \in \overline{\varphi_H}(u_C).
$$

Define a voltage assignment $\alpha : E(G) \to \Z_\ell$ by
\[
    \alpha(e)=
    \begin{cases}
        0, & e\in E(H),\\
        1, & e\in F.
    \end{cases}
\] 
Let $\widetilde G:=G(\alpha,\ell)$ and for each $i\in \Z_\ell$, let $G_i:=\widetilde G[V_i]$, where $V_i:= \{(x,i) \mid x \in V(G)\}$. 
Then, for every $i \in \Z_\ell$, $G_i\cong H$ and the lifted edge $(f_C,i)$ joins $(v_C, i)$ and $(u_C, i+1)$ for each $f_C \in F$.

We now define a proper $3$-edge-coloring $\varphi_i$ of $G_i$ in $\widetilde G$ from $\varphi_H$ by permuting the colors as follows:
\[
\begin{array}{c|c}
i & \text{permutation of colors} \\ \hline
i = 0 & 2\leftrightarrow 3\\
i\in \{1,3,\ldots,\ell-2\} & 1\leftrightarrow 2\\
i\in \{2,4,\ldots,\ell-3\} & \text{identity}\\
i = \ell-1 & 1\leftrightarrow 3
\end{array}
\]
Since color 1 is missing at $v_C$ and color 2 is missing at $u_C$ under $\varphi_H$, the missing colors at $(v_C, i)$ and $(u_C,i)$ under $\varphi_i$ are
\[
\begin{array}{c|c|c}
i & (v_C,i) & (u_C,i) \\ \hline
i = 0 & 1 & 3 \\
i\in \{1,3,\ldots,\ell-2\} & 2 & 1 \\
i\in \{2,4,\ldots,\ell-3\} & 1 & 2 \\
i = \ell-1 & 3 & 2.
\end{array}
\]
Consequently, for every $i\in \mathbb Z_{\ell}$, the vertices $(v_C, i)$ and $(u_C, i+1)$ have a common missing color: it is color 1 when $0\le i\le \ell -3$ is even, and color 2 when $1\le i\le \ell-2$ is odd, and color 3 when $i=\ell-1$.

We now color the lifted edges of $F$. Let $f_C \in F$ for some $C \in C_o$. For each $i\in \{0,1,\dots,\ell-2\}$, assign color $1$ to $(v_C,i)(u_C,i+1)$ when $i$ is even, and color $2$ when $i$ is odd. For the remaining edge $(v_C,\ell-1)(u_C,0)$, assign color $3$. Clearly, each assigned color is missing at both endpoints of the corresponding edge. Moreover, since $F \subset E(G)$ is a matching, the set $\widetilde F := \{(f,i) \mid f \in F, i\in \Z_\ell\}$ is also a matching in $\widetilde G$. Every vertex of $\widetilde G$ is incident with at most one edge of $\widetilde F$. Therefore, this procedure can be applied independently for each $f \in F$ to color all edges of $\widetilde F$. Together with the colorings $\varphi_i$ of each $G_i$, we obtain a proper $3$-edge-coloring of $\widetilde G$. Since $\widetilde G$ is cubic, Tutte's result implies that $\widetilde G$ admits a nowhere-zero $4$-flow.
\end{proof}

\section{General case -- Proof of Theorem~\ref{thm:t-3-4}}
\label{sec:main-proof}

A vertex of degree $d$  is called a \emph{$d$-vertex}.  The proof proceeds by reducing $G$ to a bridgeless cubic graph: first suppress all $2$-vertices and then replace vertices at least four by suitable paths so that the vertices on the path are   $3$-vertices. 
 After applying the cubic case to the resulting graph, we use switching to reduce the voltages on these paths to zero, allowing them to be contracted in the resulting cyclic cover.

\subsection{Reduction tools}
Let $G$ be a graph and let $v$ be a $2$-vertex of $G$ with neighbors $x$ and $y$, where $x$ and $y$ are not necessarily distinct. Suppressing $v$ means deleting $v$ and then adding a new edge joining $x$ and $y$ if $x\not = y$. If $x = y$, no new edge is added. We call this operation \emph{the suppression of $v$}. Clearly, if $G$ is bridgeless and admits a nowhere-zero $k$-flow, then the resulting graph after suppressing some $2$-vertices  remains bridgeless and admits a nowhere-zero $k$-flow as well. 

\begin{observation}\label{obs:suppression}
Let $G$ be a bridgeless graph, let $v$ be a $2$-vertex of $G$, and let $G'$ be obtained from $G$ by suppressing $v$. Then, for every integer $\ell \ge 2$ and every  $\Z_\ell$-voltage assignment $\alpha$ on $G'$, there exists a $\Z_\ell$-voltage assignment $\beta$ on $G$ such that $G'(\alpha,\ell)$ is obtained from $G(\beta,\ell)$ by suppressing all lifts of $v$.
\end{observation} 
\begin{proof}
    Let $x$ and $y$ be the neighbors of $v$.  
    We use the reference orientation inherited from $G$ on edges common to $G$ and $G'$. If $x \ne y$, let $e$ be the edge obtained by suppressing $v$, with reference orientation from $x$ to $y$. Let $G'(\alpha,\ell)$ be an $\ell$-fold cyclic cover of $G'$.

    Suppose that $x \ne y$. Define a voltage assignment $\beta: E(G) \to \Z_\ell$ by $\beta(f) = \alpha(f)$ for every edge $f$ not incident with $v$ and set $\beta(xv) = 0$. For the edge $vy$, choose its voltage so that the voltage in the direction from $v$ to $y$ is $\alpha(e)$, that is, 
    \[
    \beta(vy)=
    \begin{cases}
        \alpha(e), &\text{if the reference orientation of $vy$ is from $v$ to $y$},\\
        -\alpha(e), &\text{if the reference orientation of $vy$ is from $y$ to $v$}.
    \end{cases}
    \]
    For each $i \in \Z_\ell$, the lifted edge
$(e,i) = (x,i)(y,i+\alpha(e))$ of $G'(\alpha,\ell)$
is replaced in $G(\beta,\ell)$ by the path
$(x,i)(v,i)(y,i+\alpha(e))$.
Thus suppressing all lifts $(v,i)$ of $v$ yields the graph
 $G'(\alpha,\ell)$.

Now suppose that $x=y$. Then  $G'=G-v$.
Extend $\alpha$ to a voltage assignment
$\beta:E(G)\to\Z_\ell$ by assigning voltage $0$ to both
edges incident with $v$.
For each $i\in\Z_\ell$, the vertices $(v,i)$ and $(x,i)$
are joined by two parallel edges in $G(\beta,\ell)$.
Deleting all lifts $(v,i)$ of $v$ therefore yields the graph
 $G'(\alpha,\ell)$.
This is equivalent to suppressing all lifts of $v$.
\end{proof}

We next consider vertices of degree at least four. For such a vertex $v$ of degree $d$, by \emph{blowing up} $v$ into a $d$-cycle $C_v$, we mean replacing $v$ with $C_v$ and reattaching the edges formerly incident with $v$ bijectively to the vertices of $C_v$, leaving their other endpoints unchanged. One common adjustment to the blowing-up process is to blow up a vertex to a path $P_v$ instead of a cycle $C_v$; the following lemma formalizes this for our purposes.

\begin{lemma}\label{lem:tree-bridgeless}
Let $G$ be a bridgeless graph and let $v$ be a vertex of
degree $d \ge 4$. There exists an ordering
$e_0,e_1,\dots,e_{d-1}$ of the edges incident with  $v$
such that the graph $G_v$ obtained by replacing $v$ with
a path $P_v=x_1x_2\cdots x_{d-2}$ and reattaching
\begin{itemize}
\setlength\itemsep{-0.15cm}
    \item $e_0$ and $e_1$ to $x_1$;
    \item $e_i$ to $x_i$ for each $2 \le i \le d-3$;
    \item $e_{d-2}$ and $e_{d-1}$ to $x_{d-2}$,
\end{itemize}
while leaving the other endpoints unchanged, is bridgeless.
Moreover, every vertex of $P_v$ has degree $3$ in $G_v$.
\end{lemma}

\begin{proof} 
Since $G$ is bridgeless, it contains a cycle $C$ through $v$. Order the edges incident with $v$ as $e_0,e_1,\dots,e_{d-1}$ so that $e_0$ and $e_{d-1}$ are the two edges of $C$ incident with $v$.

Blow up $v$ into the $d$-cycle $C_v=x_0x_1\cdots x_{d-1}x_0$, reattaching $e_i$ to $x_i$ for each $0\le i\le d-1$ and leaving its other endpoint unchanged.
Denote the resulting graph by $G'_v$, and set $H=G'_v-x_0x_{d-1}$.
We show that $H$ is bridgeless by verifying that every edge of $H$ lies on a cycle.

Note that $C-\{v\}$ together with the path
$x_0x_1\cdots x_{d-1}$ yields a cycle in $H$ containing
every edge of this path.

Now let $e\in E(H)\setminus E(C_v)$.
The corresponding edge of $G$ lies on a cycle $C_e$.
If $C_e$ avoids $v$, it remains a cycle in $H$.
Otherwise, let $e_i$ and $e_j$, with $i<j$, be the two
edges of $C_e$ incident with $v$.
Then  $C_e -\{v\}$ together with the path
$x_ix_{i+1}\cdots x_j$ yields a cycle in $H$ containing $e$.

Thus every edge of $H$ lies on a cycle, so $H$ is bridgeless.

Finally, suppress the $2$-vertices $x_0$ and $x_{d-1}$
in $H$. The resulting graph is precisely $G_v$:
the path replacing $v$ is $P_v=x_1x_2\cdots x_{d-2}$,
with $e_0,e_1$ attached to $x_1$ and
$e_{d-2},e_{d-1}$ attached to $x_{d-2}$.
Since suppression preserves bridgelessness, $G_v$ is
bridgeless, as required.
\end{proof}


We recall the notion of switching equivalence for voltage assignments; see, for example, Liu, Peyerimhoff, and Vdovina~\cite{Liu2020}. 
For a function $s: V(G) \to \Z_\ell$, the \emph{switching} of a $\Z_\ell$-voltage assignment $\alpha$ by $s$ is the voltage assignment $\alpha_s$ defined by
 $$\alpha_s(e) = \alpha(e) + s(h(e)) - s(t(e)) ~\text{in $\Z_\ell$}, \qquad\text{for every } e\in E(G). $$
  Two voltage assignments $\alpha$ and $\beta$ are \emph{switching equivalent} if $\beta = \alpha_s$ for some function $s: V(G) \to \Z_\ell$. Switching-equivalent voltage assignments induce isomorphic cyclic covers~\cite{grosstucker1987}. 

Let $C = v_1v_2\dots v_mv_1$ be a cycle in $G$ and let $e_j = v_{j}v_{j+1}$, where $v_{m+1}=v_1$. The \emph{total voltage of $C$ under $\alpha$} is
$$\alpha(C) := \sum_{j=1}^{m}\epsilon_j\alpha(e_j),$$
where $\epsilon_j = 1$ if the reference orientation of $e_j$ is from $v_j$ to $v_{j+1}$, and $\epsilon_j = -1$ otherwise.

The following switching lemma is stated in terms of voltage assignments and is a special case of a result of Zaslavsky~\cite[Lemma~5.3]{ZASLAVSKY198932}. When $K$ is a spanning tree of $G$, it specializes to the spanning-tree version that is widely used in topological graph theory; see Gross and Tucker~\cite[Section~2.5.2]{grosstucker1987}. We include a proof for completeness.

\begin{lemma}\label{lem:switching-subgraph}
Let $G$ be a graph, let $\alpha:E(G) \to \Z_\ell$ be a voltage assignment, and let $K$ be a subgraph of $G$. There exists a voltage assignment $\alpha'$ switching-equivalent to $\alpha$ such that $\alpha'(e) = 0$ for every $e \in E(K)$ if and only if every cycle of $K$ has total voltage $0$ under $\alpha$. 
\end{lemma}
\begin{proof}
We first show that the total voltage of a cycle is invariant
under switching. Let $s:V(G)\to\Z_\ell$, and let
$C=v_1v_2\cdots v_mv_1$ be a cycle with edges
$e_j=v_jv_{j+1}$, where $v_{m+1}=v_1$. Then
\[
\alpha_s(C) = \sum_{j=1}^{m}
  \bigl(\epsilon_j\alpha(e_j)+s(v_{j+1})-s(v_j)\bigr) = \alpha(C)+\sum_{j=1}^{m}\bigl(s(v_{j+1})-s(v_j)\bigr) = \alpha(C).
\]
The necessity follows immediately.

For the sufficiency, suppose that every cycle of $K$
has total voltage $0$ under $\alpha$.
We first show that, when $K$ is acyclic, $\alpha$ is
switching equivalent to a voltage assignment that vanishes
on $E(K)$. We proceed by induction on $|E(K)|$.

If $E(K)=\emptyset$, there is nothing to prove.
Otherwise, choose a leaf $y$ in a nontrivial component
of $K$, let $x$ be its unique neighbor in $K$, and set
$e=xy$ and $K'=K-y$.
By the induction hypothesis, $\alpha$ is switching
equivalent to a voltage assignment $\beta$ that vanishes
on $E(K')$.

Define $s:V(G)\to\Z_\ell$ by setting $s(z)=0$ for
every $z\ne y$ and
\[
s(y)=
\begin{cases}
-\beta(e), & \text{if $e$ is oriented from $x$ to $y$},\\
 \beta(e), & \text{if $e$ is oriented from $y$ to $x$}.
\end{cases}
\]
Then $\alpha':=\beta_s$ satisfies $\alpha'(e)=0$.
Since no edge of $K'$ is incident with $y$, this switching
leaves the voltages on $E(K')$ unchanged.
Thus $\alpha'$ vanishes on $E(K)$, completing the induction.

Now let $K$ be arbitrary, and choose a spanning forest $F$
of $K$. By the acyclic case, $\alpha$ is switching equivalent
to a voltage assignment $\alpha'$ that vanishes on $E(F)$.
For each $e\in E(K)\setminus E(F)$, let $C_e$ be the unique
cycle in $F+e$. By switching invariance and our assumption,
\[
\alpha'(C_e)=\alpha(C_e)=0.
\]
All edges of $C_e$ other than $e$ belong to $F$ and hence
have voltage $0$ under $\alpha'$. It follows that
$\alpha'(e)=0$. Therefore $\alpha'$ vanishes on $E(K)$,
as required.
\end{proof}

\subsection{Proof of Theorem~\ref{thm:t-3-4}}

We now combine the preceding results to prove
Theorem~\ref{thm:t-3-4}.

\begin{proof}
We may assume that $G$ is connected, since otherwise we
can treat its components separately.
If $G$ is Eulerian, then every vertex in any $\ell$-fold
cyclic cover of $G$ has even degree, so every such cyclic cover
admits a nowhere-zero $2$-flow.
Hence we may assume that $G$ is not Eulerian.
By Observation~\ref{obs:suppression}, we may also assume
that $G$ has minimum degree at least $3$.

Applying Lemma~\ref{lem:tree-bridgeless} successively,
replace each vertex $x\in V(G)$ of degree at least four with a path $T_x$ and reattach the edges incident with $x$ as prescribed by
the lemma. For each $3$-vertex, let $T_x = \{x\}$. Denote the resulting bridgeless cubic graph
by $H$. Then
\[
K=\bigcup_{x\in V(G)}T_x
\]
is a spanning forest of $H$, and contracting each $T_x$
to a single vertex $x$ recovers $G$.
Give each edge of $H$ corresponding to an edge of $G$
the inherited reference orientation, and orient each edge of  $T_x$
so that  $T_x$ is a directed path.

Set $k=3$ if $\ell$ is even and $k=4$ if $\ell$ is odd.
By Theorems~\ref{thm:t-3-cubic} and~\ref{thm:t-4-cubic},
there exists a $\Z_\ell$-voltage assignment
$\beta:E(H)\to\Z_\ell$ such that
$H(\beta,\ell)$ admits a nowhere-zero $k$-flow.
Since $K$ is acyclic, Lemma~\ref{lem:switching-subgraph}
allows us to replace $\beta$ with a switching-equivalent
voltage assignment that vanishes on $E(K)$.
Switching preserves the isomorphism type of the cyclic
cover, so we may assume that $\beta(e)=0$ for every
$e\in E(K)$.

For each $x\in V(G)$, the lift of $T_x$ in
$H(\beta,\ell)$ consists of $\ell$ disjoint
copies $T_x^i$, indexed by $i\in\Z_\ell$, with
\[
V(T_x^i)=V(T_x)\times\{i\}.
\]
Contract each $T_x^i$ to a single vertex $(x,i)$.
We show that the resulting graph is an $\ell$-fold
cyclic cover of $G$.

For each edge $f\in E(G)$, let $f^*$ be the corresponding
edge of $H$, and define $\alpha(f)=\beta(f^*)$.
Suppose that $f$ is oriented from $x$ to $y$, and write
$a=\alpha(f)$. For each $i\in\Z_\ell$, the lifted edge
$(f^*,i)$ joins a vertex of $T_x^i$ and a vertex of
$T_y^{i+a}$. After contraction, it joins $(x,i)$ and
$(y,i+a)$, exactly as prescribed by $\alpha$.
Thus the resulting graph is $G(\alpha,\ell)$.

Since $G(\alpha,\ell)$ can be obtained by contracting some edges of $H(\beta,\ell)$ and $H(\beta,\ell)$ admits a nowhere-zero $k$-flow, $G(\alpha,\ell)$ admits a nowhere-zero $k$-flow. This proves the theorem.
\end{proof}

\subsection{An application to Cartesian bundles}

We recall the definition of Cartesian bundles following Rollov\'{a} and \v{S}koviera~\cite{ROLLOVA2012}. Let $G$ and $H$ be graphs, and fix an orientation of $G$. For each edge $e = uv \in E(G)$ oriented from $u$ to $v$, choose an automorphism $\varphi_{e} \in \operatorname{Aut}(H)$. The corresponding \emph{Cartesian bundle} $G\square^{\varphi} H$, with \emph{base} $G$ and \emph{fibre} $H$, has vertex set
$$V(G\square^{\varphi} H) = V(G) \times V(H).$$
For each $u \in V(G)$ and each edge $f = xy \in E(H)$, there is a corresponding edge joining $(u,x)$ and $(u,y)$. Moreover, for each edge $e = uv \in E(G)$ oriented from $u$ to $v$ and each $x \in V(H)$, there is an edge joining $(u,x)$ and $(v, \varphi_e(x))$. Rollov\'{a} and \v{S}koviera proved the following.

\begin{theorem}[Rollov\'{a} and \v{S}koviera~\cite{ROLLOVA2012}]\label{thm:RS}
    Every Cartesian bundle of two graphs without isolated vertices has a nowhere-zero $4$-flow. 
\end{theorem}

As an application of Theorem~\ref{thm:t-3-4}, we obtain the following result.

\begin{theorem} \label{thm:cartesian-bundle}
Let $G$ be a bridgeless graph without isolated vertices, and let $H$ be a graph. There exists a Cartesian bundle with base $G$ and fibre $H$ that admits a nowhere-zero $4$-flow, unless both of the following conditions are met:

\vspace{-0.15cm}
\begin{itemize}
\setlength\itemsep{-0.15cm}
\item $H$ has exactly one isolated vertex,

\item  $G$ admits no nowhere-zero $4$-flow.
\end{itemize}
\end{theorem}

\vspace{-0.3cm}
\begin{proof}
Let $I$ be the set of isolated vertices of $H$, let $r:= |I|$, and set $H_0 := H-I$.  

Assume that $r \not = 1$ or $G$ admits a nowhere-zero $4$-flow. We will show that there is a Cartesian bundle with base $G$ and fibre $H$ that admits a nowhere-zero $4$-flow.

We first show that we may assume that $H_0$ is an empty set, and hence $r\geq 1$. In other words, $H$ consists entirely of isolated vertices.

Consider any Cartesian bundle $\Gamma := G \square^\psi H$. Let $\Gamma_I$ and $\Gamma_0$ be the subgraphs of $\Gamma$ induced by $V(G)\times I$ and $V(G)\times V(H_0)$, respectively. Since $I$ consists of isolated vertices of $H$ and is preserved by every automorphism of $H$, no edge of $\Gamma$ joins $\Gamma_I$ and $\Gamma_0$. Thus $$\Gamma = \Gamma_I \dot\cup \Gamma_0.$$
If $H_0$ is nonempty, then $\Gamma_0$ is a Cartesian bundle with base $G$ and fibre $H_0$, and hence admits a nowhere-zero $4$-flow by Theorem~\ref{thm:RS}. Hence $\Gamma$ admits a nowhere-zero $4$-flow if and only if $\Gamma_I$ does. Since every automorphism of $H[I]$ extends to an automorphism of $H$ by fixing $H_0$ pointwise, it suffices to consider the case $H_0 = \emptyset$ and $r \ge 1$.

Suppose that $r = 1$. Then $G$ admits a nowhere-zero $4$-flow. Since $H$ consists of a single vertex, every Cartesian bundle with base $G$ and fibre $H$ is isomorphic to $G$. The result follows.

Now suppose that $r \ge 2$. We construct a Cartesian bundle with base $G$ and fibre $H$ that admits a nowhere-zero $4$-flow. Write $I = \{z_i \mid i \in \Z_r\}$. By Theorem~\ref{thm:t-3-4}, there exists a voltage assignment $\alpha: E(G) \to \Z_r$ such that the derived cyclic cover $G(\alpha,r)$ admits a nowhere-zero $4$-flow. 
For each edge $e = uv \in E(G)$ oriented from $u$ to $v$, define 
\[
\varphi_e(z_i)= z_{i+\alpha(e)} \quad \text{for each } i \in \Z_r.
\]
Since $H$ consists of isolated vertices, we have $\varphi_e \in \operatorname{Aut}(H)$. Let $\Gamma:=G\square^\varphi H$ be the Cartesian bundle determined by these automorphisms. Then, $\Gamma$ is isomorphic to $G(\alpha,r)$. Indeed, for each edge $e = uv \in E(G)$ oriented from $u$ to $v$ and each $i \in \Z_r$, the corresponding edge of $\Gamma$ is $(u,z_i)(v, \varphi_e(z_i)) = (u,z_i)(v, z_{i+\alpha(e)})$, which corresponds precisely to the edge $(e,i) = (u,i)(v,i+\alpha(e))$ of $G(\alpha, r)$. 
Hence $\Gamma \cong G(\alpha,r)$ and therefore admits a nowhere-zero $4$-flow. 
\end{proof}

\section{Concluding remarks}
\label{sec:concluding}
The results in this paper show a clear distinction between even and odd values of $\ell$. For even $\ell$, every bridgeless graph has an $\ell$-fold cyclic cover admitting a nowhere-zero $3$-flow, while for odd $\ell$, this need not be the case. This leads us to the following question.

\begin{problem}
\label{q1}
    For a fixed odd integer $\ell \ge 3$, characterize the bridgeless graphs $G$ for which there exists an $\ell$-fold cyclic cover admitting a nowhere-zero $3$-flow.
\end{problem}

 If every vertex of $G$ has even degree, then every $\ell$-fold cyclic cover
of $G$ admits a nowhere-zero $2$-flow, regardless of the parity of $\ell$. For bridgeless subcubic graphs, Theorems~\ref{thm:regular-l-2} and \ref{thm:t-3-cubic} together with Observation~\ref{obs:suppression} give a complete answer, summarized below. Recall that a graph is \emph{subcubic} if its
maximum degree is at most $3$.

\begin{theorem} 
    Let $\ell \ge 3$ be odd and let $G$ be a bridgeless subcubic graph. The following statements are equivalent:
    \vspace{-0.15cm}
    \begin{itemize}
    \setlength\itemsep{-0.15cm}
    \item[(a)] Some $\ell$-fold cyclic cover of $G$ admits a nowhere-zero $3$-flow;
    \item[(b)] Every $\ell$-fold cyclic cover of $G$ admits a nowhere-zero $3$-flow;
    \item[(c)] $G$ admits a nowhere-zero $3$-flow.
    \end{itemize}
\end{theorem}

We conclude with two results on alternating double wheels,
illustrating obstructions to nowhere-zero $3$-flows in
odd-fold cyclic covers beyond the subcubic case.

For $n \ge 3$, define the \emph{alternating double wheel} $AW_{2n}$ by $$V(AW_{2n}) = \{x,y\}\cup \{a_i, b_i \mid i \in \Z_n\} ~~\text{and}~~E(AW_{2n}) = \{a_ib_i, \ b_ia_{i+1}, \ xa_i, \ yb_i \mid i \in \Z_n\}.$$ We call the cycle $a_0b_0a_1b_1\cdots a_{n-1}b_{n-1}a_0$ the {\em rim}, and the edges incident with $x$ or $y$ the {\em spokes}. Clearly, $AW_{2n}$ is simple and bridgeless.  See Fig. \ref{fig:alternating-double-wheels} for an illustration of $AW_{12}$ and $AW_{8}$.

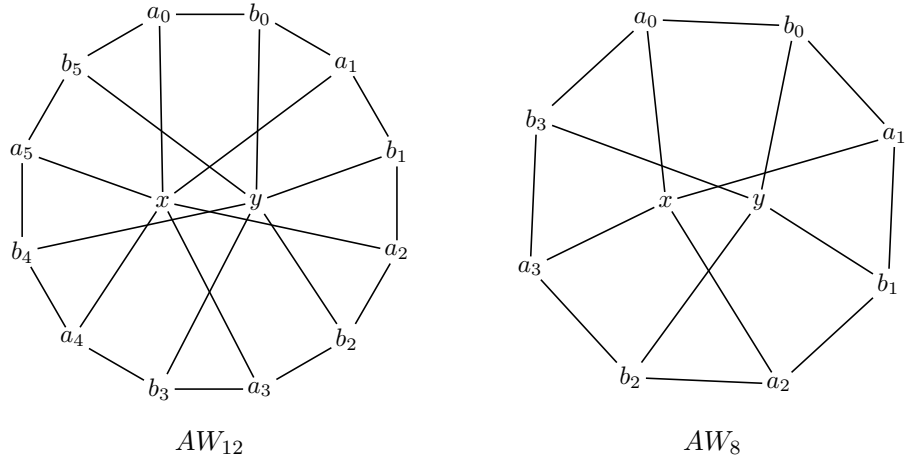
\begin{figure}[H]
\centering
\begin{tikzpicture}[
    scale = 0.95,
    transform shape,
    vertex/.style={
        inner sep=1.5pt,
        font=\small
    },
    directed/.style={
        -{Stealth[length=2mm,width=1.4mm]},
        line width=0.6pt
    },
    edge/.style={
        line width=0.6pt
    }
]
\begin{scope}[xshift=0cm]
\foreach \i in {0,...,5}{
    \pgfmathsetmacro{\aangle}{105-60*\i}
    \pgfmathsetmacro{\bangle}{75-60*\i}

    \node[vertex] (a\i) at (\aangle:2.7cm) {$a_{\i}$};
    \node[vertex] (b\i) at (\bangle:2.7cm) {$b_{\i}$};
}

\node[vertex] (x) at (-0.65,0) {$x$};
\node[vertex] (y) at ( 0.65,0) {$y$};
\foreach \i in {0,...,5}{
    \pgfmathtruncatemacro{\j}{mod(\i+1,6)}
    \draw[edge] (b\i) -- (a\i);
    \draw[edge] (b\i) -- (a\j);
}
\foreach \i in {0,...,5}{
    \draw[edge] (x) -- (a\i);
}
\foreach \i in {0,...,5}{
    \draw[edge] (b\i) -- (y);
}
\node at (0,-3.35) {$AW_{12}$};
\end{scope}
\begin{scope}[xshift=7cm]
\foreach \i in {0,...,3}{
    \pgfmathsetmacro{\aangle}{110-90*\i}
    \pgfmathsetmacro{\bangle}{65-90*\i}
    \node[vertex] (c\i) at (\aangle:2.7cm) {$a_{\i}$};
    \node[vertex] (d\i) at (\bangle:2.7cm) {$b_{\i}$};
}
\node[vertex] (xx) at (-0.65,0) {$x$};
\node[vertex] (yy) at ( 0.65,0) {$y$};
\foreach \i in {0,...,3}{
    \pgfmathtruncatemacro{\j}{mod(\i+1,4)}
    \draw[edge] (c\i) -- (d\i);
    \draw[edge] (d\i) -- (c\j);
}
\foreach \i in {0,...,3}{
    \draw[edge] (xx) -- (c\i);
}
\foreach \i in {0,...,3}{
    \draw[edge] (d\i) -- (yy);
}
\node at (0,-3.35) {$AW_8$};
\end{scope}
\end{tikzpicture}
\caption{The alternating double wheels $AW_{12}$ and $AW_8$.}
\label{fig:alternating-double-wheels}
\end{figure}

\begin{theorem}
Let $n \ge 3$. For the alternating double wheel $AW_{2n}$, the following statements are equivalent:
\vspace{-0.15cm}
\begin{itemize}
\setlength\itemsep{-0.1cm}
    \item [(a)] $3\mid n$;
    \item [(b)] $AW_{2n}$ admits a nowhere-zero $3$-flow;
    \item [(c)] There exists a $3$-fold cyclic cover of $AW_{2n}$ admitting a nowhere-zero $3$-flow. 
\end{itemize}
\end{theorem}
\begin{proof}
First show (a) implies (b).  Suppose that $3\mid n$. Orient the rim edges from a $b_i$ to an $a_j$, and orient the spokes from $x$ to the $a_i$ and from the $b_i$ to $y$. Then every vertex is a source or a sink of degree either $3$ or $n$. Since $3\mid n$, this is a modulo-$3$ orientation. This proves (b).

Now show (b) implies (c). If $(b)$ holds, assigning voltage $0$ to every edge of $AW_{2n}$ gives a $3$-fold cyclic cover consisting of $3$ disjoint copies of $AW_{2n}$. Since $AW_{2n}$ admits a nowhere-zero $3$-flow, this cyclic cover admits a nowhere-zero $3$-flow, proving $(c)$.

It remains to show that $(c) \Rightarrow (a)$. Suppose that a $3$-fold cyclic cover of $AW_{2n}$ admits a nowhere-zero $3$-flow, and fix a corresponding modulo-$3$ orientation. Every lifted rim vertex has degree $3$ and is therefore a source or a sink.

The lifted rim consists of either one cycle or three cycles. Along each such cycle, sources and sinks alternate, so all lifts of the $a_i$ on that cycle have the same type.

If all lifts of the $a_i$ have the same type, then every lift of $x$ is a source or a sink of degree $n$, forcing
$3\mid n$.

Otherwise, the lifted rim consists of three cycles. Reversing the entire orientation if necessary, we may assume that the the lifts of $a_i$ are sources on two of these cycles and sinks on the third. The total indegree over the three lifts of $x$ is then $2n$, while their total outdegree is $n$. Summing the modulo-$3$ conditions at these vertices gives $2n\equiv n\pmod3$. Hence $3\mid n$,  proving $(a)$ and completing the proof.   
\end{proof}

Thus, whenever $3 \nmid n$, $AW_{2n}$ has no $3$-fold cyclic cover admitting a nowhere-zero $3$-flow. In particular, for $n = 4$, the same conclusion holds for every odd $\ell \ge 3$.

\begin{proposition}\label{prop:D4}
    Let $\ell \ge 3$ be odd. No $\ell$-fold cyclic cover of $AW_8$ admits a nowhere-zero $3$-flow.
\end{proposition}
\begin{proof}
    Suppose, to the contrary, that some $\ell$-fold cyclic cover $\widetilde G$ of $AW_8$ admits a nowhere-zero $3$-flow. Let $\widetilde D$ be a corresponding modulo-$3$ orientation of $\widetilde G$. For each $i \in \Z_4$, let $s_i$ and $r_i$ denote the numbers of sources among the lifts of $a_i$ and $b_i$, respectively. Since every lifted rim edge joins a source and a sink, we have $s_i + r_i = \ell$ and $r_i + s_{i+1} = \ell$, and hence $s_i = s_{i+1}$ for every $i \in \Z_4$. Let $s$ be their common value. Consider the lifts of $x$. Since, for each $i \in \Z_4$, exactly $\ell-s$ lifts of $a_i$ are sinks, exactly $\ell-s$ lifted edges of $xa_i$ are directed away from the lifts of $x$. Hence the total outdegree of the lifts of $x$ is $4(\ell-s)$. On the other hand, since every lifted vertex of $x$ has degree $4$ and $\widetilde D$ is a modulo-$3$ orientation, its outdegree must be $2$. Thus the total outdegree of the lifts of $x$ is $2\ell$. This implies that $4(\ell - s) = 2\ell$, which gives $\ell = 2s$, a contradiction.
\end{proof}



\end{document}